\documentclass{birkjour}
 \newtheorem{thm}{Theorem}[section]
 \newtheorem{cor}[thm]{Corollary}
 \newtheorem{lem}[thm]{Lemma}
 
 \newtheorem{fact}[thm]{Fact}
 \theoremstyle{definition}
 
 \theoremstyle{remark}

 \numberwithin{equation}{section}

\newcommand*{\vimage}[1]{\vcenter{\hbox{\includegraphics[width=.45\textwidth]{#1}}}}
\newcommand*{\vpointer}{\vcenter{\hbox{\scalebox{2}{\Huge\pointer}}}}

\usepackage[utf8]{inputenc}
\usepackage{amsmath,amssymb,amsthm,bbm}
\usepackage{graphicx}
\usepackage{wasysym}

\DeclareMathOperator{\Prox}{Prox}
\DeclareMathOperator*{\argmin}{arg\,min}

\newcommand{\ran}{\mathbf{ran}}
\newcommand{\cone}{\mathbf{cone}}
\newcommand{\intr}{\mathbf{int}}

\newcommand{\cH}{\mathcal{H}}
\title[Cosmic Divergence]
 {Cosmic Divergence, Weak Cosmic Convergence,
 and Fixed Points at Infinity}

\begin{document}

\author[E. K. Ryu]{Ernest K. Ryu}

\address{%
7324 Mathematical Sciences,\\
Los Angeles, CA 90095\\
United States}

\email{eryu@math.ucla.edu}


\date{XXX 2017}

\begin{abstract}
To characterize the asymptotic behavior of
fixed-point iterations of non-expansive
operators with no fixed points,
Bauschke et al.\
[J.\ Fixed Point Theory Appl.\ \textbf{18:2} (2016), 297--307]
recently studied cosmic convergence and conjectured
that cosmic convergence always holds.
This paper 
presents a cosmically divergent counter example,
which disproves this conjecture.
This paper also demonstrates, with a counter example,
that cosmic convergence can be weak in infinite dimensions.
Finally, this paper
shows positive results relating to 
cosmic convergence
that provide an interpretation
of cosmic accumulation points
as fixed points at infinity.
\end{abstract}

\maketitle
\section{Introduction}
Let $\cH$ be a
(finite or infinite dimensional) Hilbert space with an inner product
$\langle\cdot,\cdot\rangle$ and induced norm $\|\cdot\|$.
Let
$T:\cH\rightarrow \cH$ be a non-expansive operator \emph{with no fixed point}.
Consider the ``fixed-point'' iteration
\begin{equation}
x^{k+1}=T(x^k)
\label{eq:fixed-iter}
\end{equation}
for some starting point $x^0\in \mathcal{H}$,
even though $T$ has no fixed point.
Since $T$ has no fixed point, 
$\{x^k\}_{k\ge0}$ diverges
in that
$\|x^k\|\rightarrow \infty$
\cite[Corollary~6(c)]{Pazy1971}.

Define
\[
\ran(I-T)=\{x-T(x)\,|\,x\in \cH\},
\]
and write $\overline{\ran(I-T)}$ for the closure of $\ran(I-T)$.
Define the \emph{minimal displacement vector} of $T$ as
\[
v=P_{\overline{\ran(I-T)}}(0)=
\argmin_{x\in \overline{\ran(I-T)}}\|x\|,
\]
which exists and is unique \cite[Lemma~4]{Pazy1971}.
Pazy, in 1971 \cite[Theorem~2]{Pazy1971}, 
elegantly characterized the asymptotic behavior of iteration \eqref{eq:fixed-iter} by showing
\[
x^k=-kv +o(k)
\]
as $k\rightarrow\infty$.
Furthermore, 
 Baillon et al., in 1978 \cite[Corollary~2.3]{BaillonBruckReich1978_asymptotic},
showed 
\[
x^{k+1}- x^k\rightarrow-v
\]
as $k\rightarrow\infty$.
When $v\ne 0$, we have $\|x^k\|=O(k)$, and these results provide a very concrete understanding
of the manner in which $x^k$ diverges to infinity.
%
%Theorem~2 of \cite{Pazy1971} and
%Assume $T$ is firmly-nonexpansive.
%The fixed-point iteration $z^{k+1}=T(z^k)$ satisfies
%\[
%z^k=-kv+o(k),\qquad
%z^{k+1}-z^k\rightarrow -v.
%\]
%
%In particular,
%$x^k=-kv +o(k)$
%and $\|x^k\|=O(k)$
%as $k\rightarrow\infty$.
%XXX explain why XXX

When $v=0$,
Pazy's result implies 
$x^{k}=o(k)$ without any characterization of direction
and
$\|x^k\|\rightarrow\infty$ at the slower rate of $\|x^k\|=o(k)$ as $k\rightarrow\infty$.
This is less concrete and less satisfactory. While there has been some work since Pazy
studying this case \cite{lins2009,gaubert2012,Bauschke2016}, 
%to further characterize asymptotic behavior of the fixed-point iteration in the case of $v=0$
our understanding is still incomplete.

In this paper, we focus on the case $v=0$;
unless stated otherwise, assume
$T:\cH\rightarrow\cH$
is non-expansive and
$0\in \overline{\ran(I-T)}\backslash \ran(I-T)$.

We say $q$ is the weak \emph{cosmic limit} of 
$\{x^k\}_{k\ge 0}\subset \cH$ if
$\|x^k\|\rightarrow\infty$ and
\[
(1/\|x^k\|)x^k\rightharpoonup q,
\]
where $\rightharpoonup$ denotes weak convergence.
We call the weak accumulation points of 
$(1/\|x^k\|)x^k$
the weak \emph{cosmic accumulation points}.
The strong cosmic limit and strong cosmic accumulation
points are defined similarly.

As a way to understand the asymptotic behavior
of the $\{x^k\}_{k\ge 0}$ sequence of 
\eqref{eq:fixed-iter},
Bauschke et al.\ \cite{Bauschke2016}
studied the cosmic accumulation points
in finite dimensional Hilbert spaces
and postulated the \emph{cosmic convergence conjecture},
which states $\{x^k\}_{k\ge 0}$ always
cosmically converges.

\subsection{Contribution and organization of this paper}
This paper furthers the study of cosmic
convergence and cosmic accumulation points
and thereby furthers the understanding
of the
asymptotic behavior
of the $\{x^k\}_{k\ge 0}$ sequence of 
\eqref{eq:fixed-iter}
for
non-expansive operators $T$ such that 
$0\in \overline{\ran(I-T)}\backslash \ran(I-T)$.

Section~\ref{s:prelim} sets up basic preliminary
information.

Section~\ref{s:count1}
presents a counter example to the cosmic convergence
conjecture. There is a (firmly) non-expansive operator $T:\mathbb{R}^2\rightarrow \mathbb{R}^2$
such that its fixed-point iteration
sequence 
does not have a cosmic limit.
This shows that the cosmic convergence conjecture is false
without any further assumptions.

Section~\ref{s:count2}
presents a non-expansive operator
on an infinite dimensional Hilbert space
for which its fixed-point iteration
sequence
cosmically converges
to $0$ weakly.
This example illustrates the additional complexity
of
analyzing cosmic convergence in infinite dimensions.

Section~\ref{s:pos}
presents positive results
towards characterizing cosmic accumulation points.
In particular, 
we show that a non-zero cosmic limit
defines a separating hyperplane
between $\overline{\ran(I-T)}$ and $\{0\}$
and can be interpreted as a fixed point at infinity.
Section~\ref{ss:existing-work}
relates the presented results to
the works of Lins
\cite{lins2009}
Gaubert and Vigeral 
\cite{gaubert2012}, 
and 
Bauschke et al.\
\cite{Bauschke2016}.
%, which also studied the the asymptotic behavior of $\{x^k\}_{k\ge 0}$ under this setup.

\section{Preliminaries}
\label{s:prelim}

% The operator $T$ has a fixed point if and only if 
% $v=0$ and $0\in \ran(I-T)$.
% This case is studied extensively, and the question of cosmic convergence
% is not really relevant anyways.

% \begin{thm}[Pazy]
% Assume $v\notin \ran(I-T)$, and consider the fixed-point iteration
% $x^{k+1}=T(x^k)$ for some starting point $x^0$.
% Then 
% \[
% \|x^k\|\rightarrow\infty
% \]
% and
% \[
% x^{k+1}-x^k\rightarrow -v.
% \]
% \end{thm}

% When $v\ne 0$, then cosmic convergence is trivially true.

% Therefore, in this paper we consider the case
% $v=0$ and $0\notin \ran(I-T)$.

% \begin{lem}
% The fixed points of $T$ is a closed convex set.
% \end{lem}
% Later we show Theorem~X, which is the analog of this theorem
% for fixed points at infinity.
Write $\rightharpoonup$ and $\rightarrow$
for weak and strong convergence, respectively.
Given a sequence $\{q^k\}_{k\ge 0}\subset \cH$, we say $q$ is a weak
accumulation point if there is a subsequence $\{q^{k_j}\}_{j\ge 0}$
such that
$q^{k_j}\rightharpoonup q$.
A strong accumulation point is defined similarly.
%In finite dimensions, we omit the words ``weak'' or ``strong'', as there is no difference.

\begin{lem}
\label{lem:cosmic-acc}
Assume two sequences $\{x^k\}_{k\ge 0}\subset \mathcal{H}$ and $\{y^k\}_{k\ge 0}\subset \mathcal{H}$ 
satisfy 
\[
\|x^k\|\rightarrow\infty,\quad\|y^k\|\rightarrow\infty
\]
as $k\rightarrow\infty$ and
\[
\|x^k-y^k\|\le C
\]
for all $k=0,1,2,\dots$ for some $C\ge 0$.
Then
$\{(1/\|x^k\|)x^k\}_{k\ge 0}$ and \\
$\{(1/\|y^k\|)y^k\}_{k\ge 0}$
have the same (weak or strong) accumulation points.
\end{lem}
\begin{proof}
The reverse triangle inequality tells us
\begin{align*}
\left|\|x^k\|-\|y^k\|\right|
&\le \|x^k-y^k\|\le C.
\end{align*}
%and the second inequality  follows from the assumption.
Dividing by $\|y^k\|$ (which is non-zero for large $k$)
we get
\[
\left|\|x^k\|/\|y^k\|-1\right|
\le C/\|y^k\|.
\]
Since $\|y^k\|\rightarrow \infty$ as $k\rightarrow\infty$,
we conclude that $\|x^k\|/\|y^k\|\rightarrow 1$.

Now we have
\begin{align*}
&\left\|
(1/\|x^k\|)x^k
-
(1/\|y^k\|)y^k
\right\|^2\\
%&\quad=
%(1/\|x^k\|)
%\left\|
%x^k
%-
%(\|x^k\|/\|y^k\|)y^k
%\right\|^2
%\\
 &\quad=
(1/\|x^k\|^2)
\left\|
x^k
-
y^k+y^k-
(\|x^k\|/\|y^k\|)y^k
\right\|^2\\
&\quad\le 
(2/\|x^k\|^2)
\left\|
x^k
-
y^k\|^2+
(2/\|x^k\|^2)\|y^k-
(\|x^k\|/\|y^k\|)y^k
\right\|^2
\\
&\quad\le 
(2/\|x^k\|^2)
C^2+
(2\|y^k\|^2/\|x^k\|^2)\left(1-\|x^k\|/\|y^k\|\right)^2
\\
&\quad\rightarrow 0.
\end{align*}
Thus we conclude 
$\{(1/\|x^k\|)x^k\}_{k\ge 0}$ and  $\{(1/\|y^k\|)y^k\}_{k\ge 0}$
have the same accumulation points.
\end{proof}

\begin{lem}
\label{lem:well-def}
Consider the fixed-point iteration \eqref{eq:fixed-iter} and its sequence $\{x^k\}_{k\ge 0}$.
The set of weak accumulation points of
$ \{(1/\|x^k\|)x^k\}_{k\ge 0}$
is the same irrespective of the starting point $x^0\in \mathcal{H}$.
\end{lem}
\begin{proof}
In addition to $\{x^k\}_{k\ge 0}$,
consider another sequence
\[
y^{k+1}=T(y^k)
\]
with some starting point $y^0\in \mathcal{H}$.
As $T$ has no fixed point, $\|x^k\|\rightarrow \infty$
and $\|y^k\|\rightarrow \infty$
\cite[Corollary~6(c)]{Pazy1971}.
By non-expansiveness of $T$, we have
\begin{align*}
 \|x^k-y^k\|\le \|x^0-y^0\|.
\end{align*}
The statement now follows from Lemma~\ref{lem:cosmic-acc}.
\end{proof}
%\begin{proof}
%In addition to $\{x^k\}_{k\ge 0}$,
%consider another sequence
%\[
%y^{k+1}=T(y^k)
%\]
%with some starting point $y^0\in \mathcal{H}$.
%As $T$ has no fixed point, $\|x^k\|\rightarrow \infty$
%and $\|y^k\|\rightarrow \infty$
%\cite[Corollary~6(c)]{Pazy1971}.
%We have
%\begin{align*}
%\left|\|x^k\|-\|y^k\|\right|
%&\le \|x^k-y^k\|\\
%&\le \|x^0-y^0\|,
%\end{align*}
%where
%the first inequality is the reverse triangle inequality,
%and the second inequality 
%follows from non-expansiveness of $T$.
%Dividing by $\|y^k\|$ (which is non-zero for large $k$)
%we get
%\[
%\left|\|x^k\|/\|y^k\|-1\right|
%\le \|x^0-y^0\|/\|y^k\|.
%\]
%Since $\|y^k\|\rightarrow \infty$ as $k\rightarrow\infty$,
%we conclude that $\|x^k\|/\|y^k\|\rightarrow 1$.
%
%This gives us
%\begin{align*}
%\limsup_{k\rightarrow\infty}
%\left\|
%x^k
%-
%(\|x^k\|/\|y^k\|)y^k
%\right\|
%&\le 
%\limsup_{k\rightarrow\infty}
%\left\|
%x^k
%-y^k
%\right\|\\
%&\le \|x^0-y^0\|.
%\end{align*}
%Since $\|x^k\|\rightarrow\infty$, we divide both sides by $\|x^k\|$ to get
%\[
%\left\|
%(1/\|x^k\|)x^k
%-
%(1/\|y^k\|)y^k
%\right\|
%\rightarrow 0
%\]
%as $k\rightarrow \infty$.
%Thus we conclude 
%$(1/\|x^k\|)x^k$ and  $(1/\|y^k\|)y^k$
%have the same accumulation points.
%\end{proof}

With Lemma~\ref{lem:well-def}
we can define
\[
Q=\text{weak accumulation points of } (1/\|x^k\|)x^k,
\]
where $\{x^{k}\}_{k\ge 0}$ is the sequence of \eqref{eq:fixed-iter}.
So $Q$ is the set of cosmic accumulation points of $T$,
and we have cosmic convergence when $Q$ is a singleton.

Note that
$Q$ may not necessarily be a subset of the unit sphere.
Write $S$ and $B$ for the unit sphere and closed unit ball
of $\cH$, i.e.,
\[
S=\{x\in \cH\,|\,\|x\|=1\},
\quad
B=\{x\in \cH\,|\,\|x\|\le1\}.
\]
\begin{fact}
\label{fact:dense}
If $\cH$ is finite dimensional,
$S$ and $B$ are compact.
If $\cH$ is infinite dimensional,
$B$ is weakly compact but $S$ is not,
and $S$ is weakly dense in $B$.
\end{fact}
The first assertion of Fact~\ref{fact:dense} follows from the Heine-Borel theorem.
The second assertion of Fact~\ref{fact:dense} can be found in standard references such as
 \cite[Corollary~6.29]{AliprantisBorder2006}.

When $\cH$ is finite dimensional
$Q\subseteq S$.
When $\cH$ is infinite dimensional
$Q\subseteq B$.
This distinction is not hypothetical;
Section~\ref{s:count2} presents an 
infinite dimensional example where
$Q\cap S$ is empty.

\begin{cor}
\label{fact:finite-dim}
$Q$ is never empty.
When $\cH$ is finite dimensional,
any $q\in Q$ satisfies $\|q\|=1$.
\end{cor}
\begin{proof}
This follows from Fact~\ref{fact:dense} 
and the fact that $Q$ is defined as the weak accumulation points of a sequence in $S$.
%$Q$ is not empty since it is defined as the weak accumulation points of a sequence in the 
%unit sphere which is weakly dense
%in the weakly compact unit ball.
%When $\cH$ is finite dimensional, 
%the unit sphere is compact and therefore the limit points have unit norm.
\end{proof}

The Kadec-Klee property, also called the Radon-Riesz property,
holds for Hilbert spaces
\cite[Proposition~3.32]{brezis2010functional}.
\begin{fact}[Kadec-Klee]
\label{lem:sw-conv}
Assume a sequence $\{q^k\}_{k\ge 0}\subset \cH$ 
satisfies $\|q^k\|=1$ and $q^k\rightharpoonup q$.
If $\|q\|=1$ then $q^k\rightarrow q$.
If $\|q\|<1$ then $q^k\nrightarrow q$.
\end{fact}
%\begin{proof}
%This follows from
%\begin{align*}
%\|q^k-q\|^2&=\|q^k\|^2+\|q\|^2-2\langle q^k,q\rangle\\
%&\rightarrow 1+\|q\|^2-2\langle q,q\rangle=1-\|q\|^2.
%\end{align*}
%\end{proof}

Weak but not strong cosmic limits arguably
tell us less about the direction in which $x^k$ diverges to infinity,
%less informative
%in understanding the asymptotic
%behavior of $\{x^k\}_{k\ge 0}$ of
%\eqref{eq:fixed-iter},
especially when the weak cosmic limit
is $0$ as in Section~\ref{s:count2}.
In finite dimensions,
Corollary~\ref{fact:finite-dim} states that
$Q\subseteq S$.
%all cosmic accumulation points are in $Q\cap S$.
In infinite dimensions,
Corollary~\ref{fact:finite-dim}
and
Fact~\ref{lem:sw-conv} states
$Q\cap S$ are the strong cosmic accumulation points 
and
$Q\backslash S$ are the weak cosmic accumulation points.

Again, this paper focuses on
non-expansive operators $T$
with minimal displacement vectors
$v=0$.
We do so, because
 strong cosmic convergence immediately
follows from known results when $v\ne 0$.
\begin{fact}
\label{lem:cosmic-strong}
Assume $T:\cH\rightarrow\cH$ is non-expansive and has no fixed point.
Write $v$ for the minimal displacement vector of $T$.
If $v\ne 0$, then 
\[
(1/\|x^k\|)x^k\rightarrow -v/\|v\|.
\]
\end{fact}
\begin{proof}
Theorem~2 of \cite{Pazy1971}
states that 
$(1/k)x^k\rightarrow -v$,
which also implies \\$(1/k)\|x^k\|\rightarrow \|v\|$.
The lemma follows from dividing the two limits.
%Theorem~1 and Lemma~4 of \cite{Pazy1971} states that
\end{proof}

\begin{cor}
\label{cor:inf-disp}
Assume $T:\cH\rightarrow\cH$ is non-expansive and has no fixed point.
If $\{x^k\}_{k\ge 0}$
does not cosmically converge
or cosmically converges weakly but not strongly,
then $v=0$.
\end{cor}
\begin{proof}
If $v\ne 0$,
then $\{x^k\}_{k\ge 0}$ cosmically converges strongly by Fact~\ref{lem:cosmic-strong}.
The corollary is the  contrapositive of this statement.
\end{proof}
By Corollary~\ref{cor:inf-disp}
%,the contrapositive ofLemma~\ref{lem:cosmic-strong},
the operators presented
in Section~\ref{s:count1} and \ref{s:count2}
have 
minimal displacement vectors $v=0$.

\section{Counter example: cosmic divergence}
\label{s:count1}
In this section, we present an operator on 
$\mathbb{R}^2$
for which cosmic convergence fails.
That this operator
has minimal displacement vector $v=0$
follows from Corollary~\ref{cor:inf-disp}.

\subsection{General construction}
Consider the function
$f:\mathbb{R}^2\rightarrow\mathbb{R}$
defined as
\[
f(x,y)=\max\{\Phi(x),\Psi(y)\}
\]
where
\[
\Phi(x)=
\int_0^x\varphi(x')\; dx',
\qquad
\Psi(y)=
\int_0^y\psi(y')\; dy'.
\]
Assume $\varphi$ and $\psi$ are
nondecreasing functions from $\mathbb{R}$ to $[-1,0)$.
This makes
$\Phi$ and $\Psi$ (continuous) strictly decreasing
convex functions from $\mathbb{R}$ to $\mathbb{R}$,
and this makes $f$ convex.
(Cf.\ Example 8.22 and Proposition 9.3 of \cite{bauschke2017}
or
Theorem~5.5 and Section 24  of \cite{rockafellar1970}.)

Consider the proximal operator
\[
\Prox_{f}(x',y')
=\argmin_{x,y\in \mathbb{R}}
\left\{
f(x,y)+\frac{1}{2}(x-x')^2+\frac{1}{2}(y-y')^2
\right\}
\]
and the iteration
\begin{equation}
(x_{k+1},y_{k+1})=\Prox_{f}(x_k,y_k),
\label{eq:fixed-iter-examp1}
\end{equation}
with starting points $x_0=y_0=0$.

\begin{lem}[\cite{moreau1965}]
The operator
\[
\Prox_f:\mathbb{R}^2\rightarrow \mathbb{R}^2
\]
is well-defined (the argmin exists and is unique) and is non-expansive.
\end{lem}

Define the curve
\[
\Gamma = \{(x,y)\in \mathbb{R}^2\,|\,
\Phi(x)=\Phi(y)\le 0\}.
\]
Note that $\Phi(0)=\Psi(0)=0$, and
thus $(x_0,y_0)\in \Gamma$.
In fact, $\{(x_k,y_k)\}_{k\ge 0}\subset \Gamma$.
\begin{lem}
\label{lem:edge}
If $\Phi(x')=\Psi(y')$ then
\[
(x,y)=\Prox_{f}(x',y')
\]
satisfies
\[
\Phi(x)=\Psi(y).
\]
\end{lem}

\begin{proof}
Without loss of generality, assume for contradiction that
$\Phi(x)>\Psi(y)$.
By definition $\Prox_f$, we have
\[
\Psi(y)< \Phi(x)=\max\{\Phi(x),\Psi(y)\}\le 
\max\{\Phi(x'),\Psi(y')\}=
\Psi(y').
\]
So $\Psi(y)< \Phi(x)\le \Psi(y')$
and, since $\Psi$ is continuous, there is a 
$\theta \in [0,1)$ such that 
\[
\Psi(\theta y+(1-\theta)y')=\Phi(x).
\]
Since 
$\Psi(y)<  \Psi(y')$, we have $y\ne y'$.
This means 
\[
\Phi(x)=
\max\{\Phi(x),
\Psi(y)
\}
=
\max\{\Phi(x),
\Psi(\theta y+(1-\theta)y')
\}
\]
and therefore
\begin{align*}
&f(x,y)+\frac{1}{2}(x-x')^2+\frac{1}{2}(y-y')^2\\
&>f(x,\theta y+(1-\theta)y')
+\frac{1}{2}(x-x')^2+\frac{1}{2}((\theta y+(1-\theta)y')-y')^2\\
&=f(x,\theta y+(1-\theta)y')
+\frac{1}{2}(x-x')^2+\frac{1}{2}\theta^2( y-y')^2.
\end{align*}
This contradicts the fact that $(x,y)=\Prox_{f}(x',y')$
is defined as the minimizer.
\end{proof}

\begin{fact}
[Theorem~2.1 of \cite{guler1991}]
\label{lem:prox-min}
The iteration \eqref{eq:fixed-iter-examp1} satisfies
\[
f(x_k,y_k)\rightarrow\inf_{x,y\in \mathbb{R}}f(x,y)
\]
as $k\rightarrow\infty$.
(To be clear, 
this holds
even when $\inf_{x,y\in \mathbb{R}}f(x,y)$ has no minimizer
and when $\inf_{x,y\in \mathbb{R}}f(x,y)=-\infty$.)
\end{fact}

%Since $\Phi(x)$ and $\Psi(y)$ are decreasing functions, the minimum of $f$ is never attained.
By Lemma~\ref{lem:edge} and Fact~\ref{lem:prox-min},
\[
f(x_k,y_k)=\Phi(x_k)=\Psi(y_k)\rightarrow \inf_{x,y\in \mathbb{R}} f(x,y)=
\max \left\{\inf_{x\in \mathbb{R}}\Phi(x),\inf_{y\in \mathbb{R}}\Psi(y)\right\}
\]
as $k\rightarrow\infty$.
Because $\Phi$ and $\Psi$ are strictly decreasing functions,
the minimum of $f(x,y)$ is not attained.
%Since $f(x_k,y_k)\rightarrow \inf_{x,y\in \mathbb{R}} f(x,y)$, 
Therefore, $\{(x_k,y_k)\}_{k\ge 0}$ has no cluster points, and
$x_k\rightarrow\infty$ and/or $y_k\rightarrow\infty$.
%
%If $\{(x_k,y_k)\}_{k\ge 0}$ had a cluster point $(x_\infty,y_\infty)$,
%then 
%\[
%f(x_\infty,y_\infty)>
%\max\{\Phi(x_\infty),\Psi(y_\infty)\}>\max \left\{\inf_{x\in \mathbb{R}}\Phi(x),\inf_{y\in \mathbb{R}}\Psi(y)\right\}=
%\inf_{x,y\in \mathbb{R}} f(x,y)
%\]
%since the strictly decreasing functions $\Phi$ and $\Psi$ have no minimizers.
%

We can think of $\Gamma$ as a 1-dimensional curve in $\mathbb{R}^2$
starting at $(0,0)$ and going off to infinity.
The sequence $\{(x_k,y_k)\}_{k\ge 0}$
stays on and moves along the curve $\Gamma$ without clustering anywhere.
Lemma~\ref{lem:stepsize} implies that
the sequence 
$\{(x_k,y_k)\}_{k\ge 0}$ traverses $\Gamma$ in small increments of size at most $1$.
Conversely, Lemma~\ref{lem:iter-approx} states that for any point $\Gamma$,
there is a point among $\{(x_k,y_k)\}_{k\ge 0}$ with distance at most $1$.

\begin{lem}[Corollary 17.19 of \cite{bauschke2017}]
\label{lem:lip}
Assume a (closed) convex function $f:\mathbb{R}^n\rightarrow\mathbb{R}$
is $L$-Lipschitz continuous.
Then any $g\in \partial f(x)$
for any $x\in \mathbb{R}^n$ satisfies
$\|g\|\le L$.
\end{lem}

%
%
%\begin{proof}
%Assume for contradiction that $g\in \partial f(x)$
%and $\|g\|>L$ for some $x\in \mathbb{R}^n$.
%Then
%\[
%f(x+\varepsilon g)-f(x)\ge \varepsilon \|g\|^2
%\]
%for any $\varepsilon>0$,
%by definition of subgradients. Then
%\[
%\frac{f(x+\varepsilon g)-f(x)}
%{\varepsilon\|g\|}\ge \|g\|>L,
%\]
%which contradicts the assumption that $f$ is $L$-Lipschitz.
%\end{proof}

\begin{lem}
\label{lem:stepsize}
%If $-1\le \varphi(x),\psi(y)$ for all $x,y\in \mathbb{R}$,
For all $x,y\in \mathbb{R}$
\[
\|(x,y)-\Prox_f(x,y)\|\le 1.
\]
\end{lem}

\begin{proof}
Since $\Phi(x)$ and $\Psi(y)$ 
are defined to be antiderivatives of $\varphi(x)$ and $\psi(y)$, functions with absolute magnitude at most $1$,
they are $1$-Lipschitz continuous functions.
This makes $f(x,y)=\max\{\Phi(x),\Psi(y)\}$ $1$-Lipschitz continuous as well.
By Lemma~\ref{lem:lip},
any $g\in \partial f(x,y)$ for any $x,y\in \mathbb{R}$
satisfies $\|g\|\le 1$.

Write
\[
(x_+,y_+)=\Prox_f(x,y).
\]
Then the optimality conditions of the minimization defining
$\Prox_f$ gives us
\[
g= (x-x_+,y-y_+)\in \partial f(x_+,y_+).
\]
 Since $\|g\|\le 1$, we have
$\|(x-x_+,y-y_+)\|\le 1$.
\end{proof}

\begin{lem}
\label{lem:iter-approx}
For any $(\xi,\zeta)\in \Gamma$, there is a $k\ge 0$  
such that 
\[
\|(\xi,\zeta)-(x_k,y_k)\|\le 1.
\]
\end{lem}
\begin{proof}
Fact~\ref{lem:prox-min} tells us
\[
f(x_k,y_k)\rightarrow 
\max\left\{\inf_{x\in \mathbb{R}}
\Phi(x),
\inf_{y\in \mathbb{R}}
\Psi(y)\right\}
\]
as $k\rightarrow\infty$.
Since $\Phi$ and $\Psi$ are strictly decreasing,
$f(x_k,y_k)$ never attains the limit for finite $k$.
%and, by construction of $\Phi$ and $\Psi$, the infimums are not attained.
Since $(\xi,\zeta)\in \Gamma$, we have $\Phi(\xi)=\Psi(\zeta)\le 0$.
Finally, since $f(x_0,y_0)=0$,
there is a $k$ such that 
\[
f(x_k,y_k)\ge 
\Phi(\xi)=\Psi(\zeta)>f(x_{k+1},y_{k+1}).
\]
By Lemma~\ref{lem:edge}, this further implies 
\begin{align*}
\Phi(x_k)&\ge 
\Phi(\xi)>\Phi(x_{k+1})\\
\Psi(y_k)&\ge 
\Psi(\zeta)>\Psi(y_{k+1}).
\end{align*}
Since $\Phi$ and $\Psi$ are strictly decreasing univariate functions,
we have  $ x_k\le \xi<x_{k+1}$ and $y_k\le \zeta< y_{k+1}$.
With Lemma~\ref{lem:stepsize}
we have 
\[
\|(x_k,y_k)-(\xi,\zeta)\|^2\le \|(x_k,y_k)-(x_{k+1},y_{k+1})\|^2\le 1.
\]
\end{proof}

\subsection{Design of $\varphi$ and $\psi$}
So the sequence $\{(x_k,y_k)\}_{k\ge 0}$ 
stays on and moves along the curve $\Gamma$.
We now design $\varphi$ and $\psi$ so that
there a sequence $\{(\xi_n,\zeta_n)\}_{n\ge 0}\subset \Gamma$
with more than one cosmic accumulation point.
Roughly speaking, this means $\Gamma$ goes off to infinity with no single cosmic direction,
and this will imply $\{(x_k,y_k)\}_{k\ge 0}$ has no single cosmic limit.
This argument is made formal in Section~\ref{ss:cosmic-conclusion}.

%    We first define $\{\xi_n\}_{n\ge 0}$ and $\{\zeta_n\}_{n\ge 0}$,
%    which increase so rapidly that
%    their increments are much larger compared their past value.
%    We  and then use 
%    them to construct $\varphi$ and $\psi$.

Let
\[
\xi_n=\sum^n_{i=1}i^i
\]
for $n=1,2,\dots$ and $\xi_0=0$.
Also let
\[
\zeta_n=\sum^n_{i=1}\frac{2}{3+(-1)^i}i^i
\]
for $n=1,2,\dots$ and $\zeta_0=0$.
For $n>1$, we have
\[
\xi_n=\sum^n_{i=1}i^i\le \sum^n_{i=1}n^i
\le \sum^n_{i=0}n^i=\frac{n^{n+1}-1}{n-1}\le \frac{n^{n+1}}{n-1}
\]
and 
\[
\zeta_n\le \xi_n\le \frac{n^{n+1}}{n-1}.
\]
For $n>1$, we have
\[
0\le 
\frac{\xi_n}{\xi_{n+1}-\xi_n}
=
\frac{\xi_n}{(n+1)^{n+1}}\le 
\frac{n^{n+1}}{(n-1)(n+1)^{n+1}}\le \frac{1}{n-1},
\]
where the first inequality and equality follows from the definition of $\{\xi_n\}_{n\ge 0}$,
the second inequality follows from the bound on $\xi_n$, and the third inequality
follows from $n^{n+1}\le (n+1)^{n+1}$. 
Likewise, we have
\[
0\le
\frac{\zeta_n}{\zeta_{n+1}-\zeta_n}
\le 
\frac{\zeta_n}{(n+1)^{n+1}/2}\le 
\frac{2n^{n+1}}{(n-1)(n+1)^{n+1}}\le \frac{2}{n-1}
\]
for $n>1$.
These inequalities say that the increment from $\xi_n$ to $\xi_{n+1}$
is much larger than $\xi_n$ (and the same is true for $\zeta_n$).

Define
\[
\varphi(x)=\left\{
\begin{array}{ll}
-1&\text{for }x<0\\
-\frac{1}{n^n}&
\text{for }x\in [\xi_{n-1},\xi_{n}),\quad \text{for }n=1,2,\dots
\end{array}
\right.
\]
and
\[
\psi(y)=\left\{
\begin{array}{ll}
-1&\text{for }y<0\\
-\frac{3+(-1)^n}{2}
\frac{1}{n^n}&
\text{for }y\in [\zeta_{n-1},\zeta_{n}),\quad \text{for }n=1,2,\dots.
\end{array}
\right.
\]
Note that 
\begin{align*}
\Phi(\xi_n)&=
\sum^n_{i=1}
\int ^{\xi_i}_{\xi_{i-1}}\varphi(x')\;dx'\\
&=
\sum^n_{i=1}
\left(-\frac{1}{i^i}(\xi_i-\xi_{i-1})\right)\\
&=
\sum^n_{i=1}
(-1)=-n,
\end{align*}
and, by the same reasoning, 
$\Psi(\zeta_n)=-n$.
Therefore, 
$\Phi(\xi_n)=\Psi(\zeta_n)=-n$
and $(\xi_n,\zeta_n)\in \Gamma$
for $n=1,2,\dots$.

The sequence $\{(\xi_n,\zeta_n)\}_{n\ge 0}$
diverges to infinity, i.e.,
$\|(\xi_n,\zeta_n)\|\rightarrow\infty$,
and has two cosmic accumulation points.
To see why, note that 
\begin{align*}
\frac{\xi_{n+1}}{\zeta_{n+1}}&=
\frac{\xi_{n+1}
-\xi_n+\xi_n
}{\zeta_{n+1}
-\zeta_n+\zeta_n
}\\
&=
\frac{\xi_{n+1}
-\xi_n
}{\zeta_{n+1}
-\zeta_n
}
\frac{
1+\xi_n/(\xi_{n+1}-\xi_n)
}{
1+\zeta_n/(\zeta_{n+1}-\zeta_n)
}\\
&=
\frac{3+(-1)^{n+1}}{2}
\frac{
1+\mathcal{O}(1/n)
}{
1+\mathcal{O}(1/n)
}.
\end{align*}
So
\[
\frac{\xi_{2m}}{\zeta_{2m}}
\rightarrow 2\qquad
\frac{\xi_{2m+1}}{\zeta_{2m+1}}
\rightarrow 1
\]
as $m\rightarrow\infty$.
This tells us  $\{(\xi_n,\zeta_n)\}_{n=0,2,4,\dots}$
has the cosmic limit $(2/\sqrt{5},1/\sqrt{5})$
and 
$\{(\xi_n,\zeta_n)\}_{n=1,3,5,\dots}$
has the cosmic limit $(1/\sqrt{2},1/\sqrt{2})$.
We conclude that the whole sequence
$\{(\xi_n,\zeta_n)\}_{n\ge 0}$ has two cosmic accumulation points.

\subsection{Cosmic divergence}
\label{ss:cosmic-conclusion}
Since $\{(\xi_n,\zeta_n)\}_{n\ge 0}\subset \Gamma$,
Lemma~\ref{lem:iter-approx} says
there is a subsequence
\[
\{(x_{k_n},y_{k_n})\}_{n\ge 0}\subset \{(x_{k},y_{k})\}_{k\ge 0}
\]
such that 
\[
\|(x_{k_n},y_{k_n})-(\xi_n,\zeta_n)\|\le 1
\]
for $n=0,1,\dots$.
By Lemma~\ref{lem:cosmic-acc},
$\{(x_{k_n},y_{k_n})\}_{n\ge 0}$ has the same
cosmic accumulation points
as $(\xi_n,\zeta_n)$.
Thus  we conclude that
the whole sequence $\{(x_{k},y_{k})\}_{k\ge0}$ has more than $1$ cosmic accumulation point.

\section{Counter example: weak, but not strong, cosmic convergence}
\label{s:count2}
In this section, we present an operator on 
an infinite dimensional Hilbert space
for which its fixed-point iteration
cosmically converges to $0$ weakly, i.e.,
the iterates of \eqref{eq:fixed-iter} satisfy
\[
(1/\|x^{k}\|)x^k\rightharpoonup 0.
\]
This example shows that weak cosmic convergence
is indeed possible in infinite dimensions
and that
a weak cosmic limit 
may not be very informative
in understanding the direction in which $\{x^k\}_{k\ge 0}$ diverges to infinity.
That this operator
has minimal displacement vector $v=0$
follows from Corollary~\ref{cor:inf-disp}.

\subsection{Univariate building block}
We first build and analyze a non-expansive operator on $\mathbb{R}$
and later use it to build the full example.

Consider the convex function $\varphi:\mathbb{R}\rightarrow \mathbb{R}$
defined as
\begin{equation}
\varphi(x)=
\left\{
\begin{array}{ll}
1-x&\text{for } x<0\\
e^{-x}&\text{for }x \ge 0.
\end{array}
\right.
\label{eq:phi-def}
\end{equation}
Let $\alpha\in (0,1]$. Then the gradient descent operator
\[
t(x)=x-\alpha \varphi'(x)
\]
is non-expansive.

Consider the fixed-point iteration
\begin{align}
x_{k+1}&=t(x_k)\nonumber\\
&=x_k+\alpha e^{-x_k}\label{eq:fixed-counter}
\end{align}
with starting point $x_0=0$.
It is easy to see that $x_k\ge 0 $ for all $k=0,1,\dots$.
The iterates satisfy
\[
x_k=\log k+\mathcal{O}(1).
\]
\begin{lem}
\label{lem:iter-asymp}
For $k=0,1,\dots$, the $x_k$ iterates of \eqref{eq:fixed-counter}
satisfy
\[
\log  (k+1)+\log \alpha\le x_k
\le \log  (k+1)+\log 2.
\]
\end{lem}

\begin{proof}
We prove this by induction.
Since $\alpha\in (0,1]$ and $x_0=0$, the base case
holds trivially.

Now assume the statement holds for $k$. Then
\begin{align*}
x_k&\ge \log \alpha + \log(k+1)\\
x_{k+1}&\ge\log \alpha + \log(k+1)+\frac{1}{k+1}\\
&\ge \log \alpha+\log(k+2).
\end{align*}
The second line follows from plugging the first line into
the function  $x+\alpha e^{-x}$,
which is nondecreasing on $x\in[\log \alpha,\infty)$.
The third line follows from the inequality
\[
\log (k+1)+\frac{1}{k+1}\ge
\log(k+2),
\]
which in turn follows from 
upper bounding the concave function $\log$ 
with its first order Taylor expansion about $k+1$.

Similarly, we have
\begin{align*}
x_k&\le \log(k+1)+\log 2\\
x_{k+1}&\le \log(k+1)+\log 2+\frac{\alpha}{2(k+1)}\\
&\le \log(k+1)+\log 2+\frac{1}{2(k+1)}\\
&\le \log(k+1)+\log 2+\frac{1}{k+2}\\
&\le \log(k+2)+\log 2.
\end{align*}
Again, the second line follows from plugging the first line into
$x+\alpha e^{-x}$.
The third line simply follows from $\alpha\le 1$.
The fourth line follows from the inequality 
\[
\frac{1}{2(k+1)}\le 
\frac{1}{k+2}
\]
for $k\ge 0$.
The fifth line follows from 
\[
 \log (k+2)-\frac{1}{k+2}\ge \log (k+1),
\]
which in turn follows from 
upper bounding the concave function $\log$ 
with its first order Taylor expansion about $k+2$.
\end{proof}

\subsection{Main construction}
Consider the standard $\ell^2$ space, the space of 
square-summable (real) sequences with inner product
\[
\langle x,y\rangle =\sum^\infty_{i=1}x_iy_i.
\]

Consider the convex function $f:\ell^2\rightarrow \mathbb{R}$
\[
f(x)=\sum^\infty_{i=1}\frac{1}{i^2}\varphi(x_i),
\]
where $\varphi$ is as defined in \eqref{eq:phi-def}.
Consider the operator
\[
T(x)=x-\nabla f(x)
\]
which is
\[
(T(x))_i=x_i-\frac{1}{i^2}\varphi'(x_i)
\]
for $i=1,2,\dots$.
Since the mapping of each coordinate $x_i\mapsto x_i-(1/i^2) \varphi'(x_i)$ is non-expansive,
$T:\ell^2\rightarrow \ell^2$ is non-expansive.

Now consider the fixed-point iteration
\[
x^{k+1}=T(x^k)
\]
with starting point $x^0=0\in \ell^2$.
This simplifies to
\[
(x^{k+1})_i=x^k_i+(1/i^2)e^{-x^k_i}
\]
for $i=1,2,\dots$.
%It is easy to see that $\|x^k\|> 0$ for $k>0$.
Lemma~\ref{lem:iter-asymp} states
\[
x^k_i=\log k+\mathcal{O}(1)
\]
as $k\rightarrow\infty$ for fixed $i$.

This coordinate-wise rate implies 
\[
(1/\|x^k\|)x^k_i\rightarrow 0
\]
as $k\rightarrow\infty$ for fixed $i$.
To see why, pick any $N\in \mathbb{N}$ and  we get
\[
\|x^k\|^2\ge (|x^k_1|^2+\dots+|x^k_N|^2)\ge N(\log k)^2+\mathcal{O}(1).
\]
Therefore
\[
\limsup_{k\rightarrow\infty}(1/\|x^k\|)|x^k_i|\le 1/\sqrt{N},
\]
and we let $N\rightarrow\infty$ to conclude $(1/\|x^k\|)x^k_i\rightarrow 0$.
% as $k\rightarrow\infty$ for fixed $i$.
Since each coordinate of the sequence $\{(1/\|x^k\|)x^k\}_{k\ge 0}\subset \ell^2$ 
converges to $0$, we have
\[
(1/\|x^k\|)x^k\rightharpoonup 0.
\]

On the other hand, 
Fact~\ref{lem:sw-conv} and
the fact that $(1/\|x^k\|)x^k$
has unit norm for $k>0$ implies
\[
(1/\|x^k\|)x^k\not\rightarrow 0.
\]
Therefore, $\{x^k\}_{k\ge 0}$ 
cosmically converges to $0$ weakly but not strongly.

\section{Positive results}
\label{s:pos}
In this section, we presents positive results
towards characterizing cosmic accumulation points.
Section~\ref{ss:res}
states the results
and Section~\ref{ss:fixed-pt-inf}
provides their interpretation.
Section~\ref{ss:existing-work}
relates these results to past work.

\subsection{Results}
\label{ss:res}
\begin{thm}
\label{thm:sep-hyp}
For any $q\in Q$ and $x\in \cH$, 
\[
\langle T(x)-x,q\rangle \ge 0.
\]
So $q$
defines a separating hyperplane
between $\overline{\ran(I-T)}$ and $\{0\}$,
if $q\ne 0$.
\end{thm}
\begin{proof}
By definition of $Q$, 
the iterates
$\{x^k\}_{k\ge 0}$ of \eqref{eq:fixed-iter}
has a subsequence
$\{x^{k_j}\}_{j\ge 0}$ such that
\[
(1/\|x^{k_j}\|)x^{k_j}\rightharpoonup q,
\quad
\|x^{k_j}\|\rightarrow\infty,
\quad
\|T(x^{k_j})-x^{k_j}\|\rightarrow 0.
\]
For any $x\in \cH$, we have
\begin{align*}
\|x-x^{k_j}\|^2
&\ge
\|T(x)-T(x^{k_j})\|^2\\
&=
\|x-x^{k_j}\|^2
+
2\langle
x-x^{k_j},
T(x)-x-T(x^{k_j})+x^{k_j}
\rangle\\
&
\qquad
+
\|T(x)-x-T(x^{k_j})+x^{k_j}\|^2.
\end{align*}
We reorganize this to get
\begin{align*}
&-\frac{1}{\|x^{k_j}\|}
\|T(x)-x-T(x^{k_j})+x^{k_j}\|^2\\
&\qquad\qquad\qquad\ge\frac{2}{\|x^{k_j}\|}
\langle
x-x^{k_j},
T(x)-x-T(x^{k_j})+x^{k_j}
\rangle.
\end{align*}
By taking the limit $j\rightarrow\infty$ we get
\[
0
\ge
-\langle
q,
T(x)-x
\rangle.
\]
\end{proof}

\begin{cor}
\label{cor:main}
If $q\in Q$ and $q\ne 0$, then
the iterates of \eqref{eq:fixed-iter} satisfy
\[
\langle x^{k},q\rangle \rightarrow\infty
\]
monotonically as $k\rightarrow\infty$
for any starting point
$x^0\in \cH$.
\end{cor}
\begin{proof}
Since
\[
\langle x^{k+1}-x^k,q\rangle
\ge 0
\]
by Theorem~\ref{thm:sep-hyp},
the scalar sequence
$\{\langle x^{k},q\rangle\}_{k\ge 0} $
is nondecreasing.
Since
$q\in Q$, there is a subsequence
$\{x^{k_j}\}_{j\ge 0}$ such that 
\[
(1/\|x^{k_j}\|)x^{k_j}\rightharpoonup q.
\]
Therefore, we have
\[
(1/\|x^{k_j}\|)\langle x^{k_j},q\rangle\rightarrow 
\langle q,q\rangle
\]
and $\langle x^{k_j},q\rangle\rightarrow\infty$ as $j\rightarrow\infty$.
Since 
$\{\langle x^{k},q\rangle\}_{k\ge 0} $
is nondecreasing, and since a subsequence of it converges to $\infty$, the whole sequence converges to $\infty$.
\end{proof}

\begin{cor}
If $q_1,q_2\in Q$ then
\[
\langle q_1,q_2\rangle \ge 0.
\]
\end{cor}
\begin{proof}
Assume $q_1\ne 0$ as otherwise the result is trivial.
Since
$q_2\in Q$, there is a subsequence
$\{x^{k_j}\}_{j\ge 0}$ such that 
\[
(1/\|x^{k_j}\|)x^{k_j}\rightharpoonup q_2.
\]
By Corollary~\ref{cor:main}, we have
\[
\langle q_1,x^{k_j}\rangle \rightarrow
\infty
\]
and therefore
\[
\lim_{j\rightarrow\infty}
\frac{1}{\|x_{k_j}\|}
\langle q_1,x^{k_j}\rangle
=
\langle q_1,q_2\rangle \ge 0.
\]
\end{proof}

Given a set $A\subseteq \cH$, write $A^*$ for the dual cone of $A$, i.e.,
\[
A^*=\{v\in \cH\,|\,\langle v,a\rangle \ge 0\,\text{for any }a\in A\}.
\]
\begin{thm}
\label{thm:main}
If $T:\cH\rightarrow\cH$ is non-expansive,
then
\[
 Q\subseteq 
 \overline{\cone (\ran(T-I))}\cap
 (\overline{\cone (\ran(T-I))})^*
 \subseteq 
   Q^*.
\]
\end{thm}
\begin{proof}
That $Q\subseteq \overline{\cone (\ran(T-I))}$ follows from Theorem~2.6 of \cite{Bauschke2016}.
That $Q\subseteq (\overline{\cone (\ran(T-I))})^*$
follows from Theorem~\ref{thm:sep-hyp}.
The final inclusion
follows from the fact that
when $B$ is cone, 
\[
A\subseteq B\cap B^*
\]
implies
$B^*\subseteq A^*$ and $B\subseteq A^*$,
cf.\ Proposition~6.24 of \cite{bauschke2017}.
\end{proof}

\subsection{Interpretation via points at infinity}
\label{ss:fixed-pt-inf}
In addition to the interpretation
as a separating hyperplane,
a $q\in Q$ such that $q\ne 0$
has an interpretation 
as a fixed point at infinity.

In the classical case
where $S:\cH\rightarrow\cH$
is non-expansive and has fixed point $x^\star\in \cH$, we have
\[
\|S(x)-x^\star\|\le \|x-x^\star\|
\]
for any $x\in \cH$.
In other words,
$S(x)$ stays within the 
closed ball of radius
$\|x-x^\star\|$
centered about
$x^\star$.

We have an analogous
geometric interpretation of
Theorem~\ref{thm:sep-hyp}:
$T(x)$ stays within the 
half-space
\[
\{y\in \cH\,|\,
\langle q,y\rangle \ge \langle q,x\rangle\},
\]
for any $x\in \cH$.
See Figure~\ref{fig:mapping}.
This is like saying
$T(x)$ stays within the 
closed ball of radius
$\infty$
centered about
$\infty \cdot q$, and this is the intuition
of the proof of Theorem~\ref{thm:sep-hyp}.
We can say all steps of the fixed-point iteration
\eqref{eq:fixed-iter}
proceeds forward infinitely
in the direction $q$
by Corollary~\ref{cor:main}.

Write 
$B\subset\cH$ for the closed unit ball,
$\intr B\subset\cH$ for the open unit ball,
and $S\subset \cH$ for the unit sphere.
Consider the one-to-one map
$f: \cH\rightarrow \intr B$
defined as
\begin{equation}
f(x)=\frac{1}{1+\|x\|}x.
\label{eq:ball-map}
\end{equation}
It is interesting to map the geometric interpretation
in $\cH$ onto $\intr B$ via $f$.
%In Figure~\ref{fig:mapping}, the geometric interpretation mapped under $f$.

In Figure~\ref{fig:mapping}, 
the half-space is mapped
into a crescent-shaped image under $f$,
and $\infty\cdot q$ is mapped to $\lim_{a\rightarrow\infty}f(a q)\in S$,
the outer center of the crescent shape.
In the fixed-point iteration \eqref{eq:fixed-iter},
the half-space defined by $x^k$ moves forward
in the direction $q$ in $\cH$.
This corresponds to the crescent shape
becoming thinner, but not smaller, in $B$.

\begin{figure}
\begin{center}
$\vimage{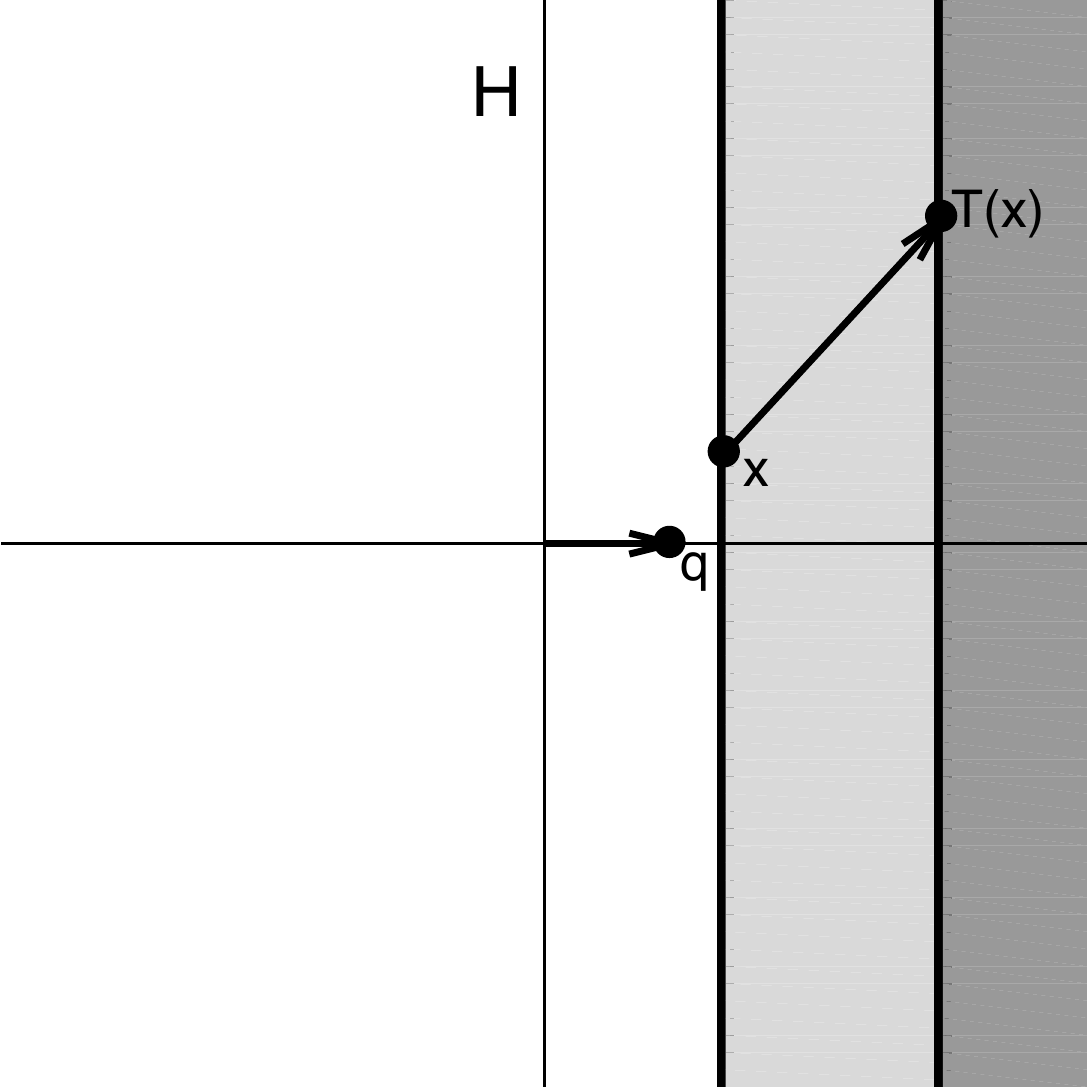}
\vpointer
\vimage{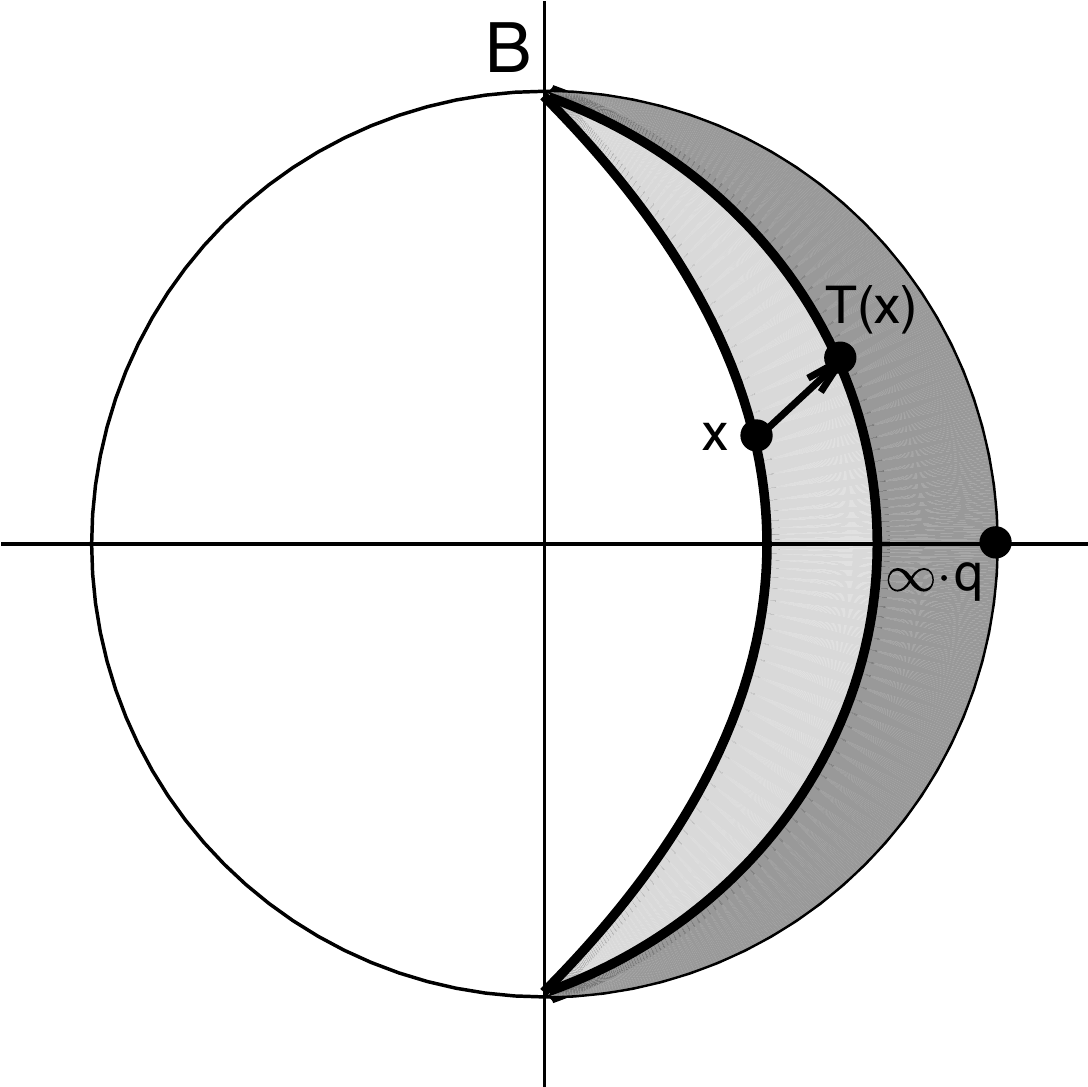}$
\end{center}
\caption{
Geometric interpretation Theorem~\ref{thm:sep-hyp}
on $\mathcal{H}$ and on $B$.
The right is the image of the left under $f$ as defined in \eqref{eq:ball-map}.
}
\label{fig:mapping}
\end{figure}

\subsection{Relationship to existing work}
Lins, Gaubert, and Vigeral presented results
related to Theorem~\ref{thm:sep-hyp}.
Under the setup and notation of this paper,
they are as follows.
\label{ss:existing-work}
\begin{thm}[Theorem~1.2 of \cite{lins2009}]
\label{thm:lins}
Let $T:\mathbb{R}^n\rightarrow\mathbb{R}^n$ be non-expansive
with no fixed points.
Then there is a $\nu\in \mathbb{R}^n$ such that
\[
\langle \nu,x^k\rangle\rightarrow \infty
\]
for any starting point $x^0$.
\end{thm}
\begin{thm}[Corollary~21 of \cite{gaubert2012}]
\label{thm:gaubert}
Let $T:\mathbb{R}^n\rightarrow\mathbb{R}^n$ be non-expansive 
with no fixed points.
Then there is a $\nu\in \mathbb{R}^n$ such that
\[
\langle \nu,x^0\rangle\le
\langle \nu,x^1\rangle\le
\langle \nu,x^2\rangle\le
\dots
\]for any starting point $x^0$.
\end{thm}

In their original papers,
Theorem~\ref{thm:lins} and \ref{thm:gaubert}
are more general,
as they hold 
for finite dimensional Banach spaces.
In the case of finite dimensional Hilbert spaces,
Corollaries~\ref{cor:main} and \ref{fact:finite-dim} subsumes 
Theorem~\ref{thm:lins} and \ref{thm:gaubert}.
Furthermore, Corollary~\ref{cor:main}
is more constructive than Theorem~\ref{thm:lins} and \ref{thm:gaubert},
as it provides an explicit characterization of
the linear functional $\nu$.

Bauschke, Douglas, and Moursi presented
various results towards understanding
the conditions for cosmic convergence, including
Theorem~\ref{thm:bauschke}.
The presented Theorem~\ref{thm:main} builds on and strengthens
Theorem~\ref{thm:bauschke}.
\begin{thm}[Theorem~2.6 of \cite{Bauschke2016}]
\label{thm:bauschke}
If $T:\cH\rightarrow \cH$ is non-expansive, then
\[
 Q\subseteq 
 \overline{\cone (\ran(T-I))}.
\]
\end{thm}

\section{Conclusion}
We discussed 
counter examples showing that
the cosmic convergence conjecture is false
and that weak cosmic convergence can happen
in infinite dimensions.
We then discussed some 
new positive results relating to 
cosmic convergence, which allow us to
interpret cosmic accumulation points
as fixed points at infinity.

Despite the counter examples, however,
the author still believes
the notion of cosmic convergence is useful
in understanding the asymptotic behavior
of the fixed-point iteration \eqref{eq:fixed-iter}.
In fact, the author still conjectures that cosmic convergence usually holds.
Finding general conditions
that ensure cosmic convergence
is an interesting future direction.

One potentially
interesting and fruitful
approach may be to focus on firmly
non-expansive operators.
Many operators that arise from convex optimization are not just non-expansive
but are furthermore firmly non-expansive.
At the same time, however, 
firm non-expansivity 
will not make the question of cosmic convergence trivial.
After all,
the counter examples of Section~\ref{s:count1}
and \ref{s:count2}
are firmly non-expansive operators.
Also it does not seem like
we can strengthen Theorem~\ref{thm:sep-hyp}
by assuming firm non-expansivity.

\end{document}